\documentclass[12pt, a4paper]{amsart}

\usepackage{geometry} 
\usepackage{fancyhdr}
\usepackage{url}
\usepackage{amssymb,amsthm}
\usepackage{amsmath}
\usepackage{accents}
\usepackage{eucal}
\usepackage{amscd}
\usepackage[shortalphabetic]{amsrefs}
\usepackage{mathptmx}
\usepackage{times}
\usepackage{graphics}
\usepackage{units}
\usepackage{enumitem}

\allowdisplaybreaks[1]

\newtheorem{theorem}{Theorem}[section]
\newtheorem{lemma}[theorem]{Lemma}
\newtheorem{corollary}[theorem]{Corollary}

\newcommand{\circo}{\accentset{\circ}}

\providecommand{\abs}[1]{\lvert#1\rvert}

\DeclareMathOperator{\tr}{tr}

\newcommand{\ho}{\accentset{\circ}{h}}

\newcommand{\p}{\partial}

\newcommand{\e}{\epsilon}

\newcommand{\mc}{\mathcal}
\newcommand{\R}{\mathbb{R}}
\newcommand{\mbb}{\mathbb}

\DeclareMathOperator{\I}{I}
\DeclareMathOperator{\II}{II}

\newcommand{\Sc}{Sc}

\begin{document}

\title[Ancient Solutions to High Codimension MCF]{Pinched Ancient Solutions to the High Codimension Mean Curvature Flow}

\author{Stephen Lynch}
\address{Freie Universit\"{a}t Berlin\\ Arnimallee 3\\ Berlin 12053\\Germany}
\email{stephen.lynch@fu-berlin.de}
\author{Huy The Nguyen}
\address{Queen Mary University of London\\ Mile End Road\\ London E1 4NS\\ United Kingdom}
\email{h.nguyen@qmul.ac.uk}

\begin{abstract}
We study solutions of high codimension mean curvature flow defined for all negative times, usually referred to as ancient solutions. We show that any compact ancient solution whose second fundamental form satisfies a certain natural pinching condition must be a family of shrinking spheres. Andrews and Baker \cite{Andrews2010} have shown that initial submanifolds satisfying this pinching condition, which generalises the notion of convexity, converge to round points under the flow. As an application, we use our result to simplify their proof. 
\end{abstract}

\maketitle

\section{Introduction}
In this paper, we consider ancient solutions to the mean curvature flow with pinched second fundamental form. A family of smooth immersions $ F : \mathcal M ^{n}\times ( t_0 ,t _1) \to \R ^{n+ k}$ is a solution to the mean curvature flow if 
\begin{align*}
\partial_ t F ( x ,t ) =  {H}(x ,t), \quad x \in \mathcal M, \;t \in ( t _ 0, t_1)   ,
\end{align*}   
where $ H (x, t)$ denotes the mean curvature vector. We will always assume that $ n\geq 2$, $k \geq 1$, and that $ \mathcal M$ is a complete smooth manifold. A solution is referred to as ancient if $t_0 = -\infty$.  

The mean curvature flow is (weakly) parabolic and hence ill-posed backwards in time, however ancient solutions are interesting for several reasons. They arise naturally as tangent flows near singularities \cite{Hamilton1995}, \cite{White2000}, \cite{White2003}, \cite{HuSi99a}, \cite{HuSi09}, and are therefore models for singularity profiles of the flow \cite{HuSi99a}.

Solutions that shrink homothetically provide the first of many examples of ancient solutions. Writing $\mc M_t := F(\mc M,t)$, we have the family of round spheres $\mc M _ t = \mbb S^n_{R(t) }$ with $ R(t) = \sqrt{- 2 n t } $ and the shrinking cylinders $\mc M _t = \mbb S^{n-m} _{R(t) }\times \R^m$, $R(t) = \sqrt{ -2 (n-m) t}$. The Angenent oval \cite{Angenent1992} (also known as the paperclip solution \cite{Lukyanov2004}) is an example of a non-homothetically shrinking compact ancient solution to the curve shortening flow, obtained by gluing together two copies of the grim reaper (or the hairpin solution \cite{Bakas2007}). The grim reaper itself, the bowl solitons and other translating solutions are not only ancient but eternal ($t_1 = \infty$). In \cite{Haslhofer2016}, Haslhofer and Hershkovits construct convex ancient solutions which flow from each of the cylinders $\mbb S^{n-m} \times \mbb R^m$ at $t_0 = -\infty$ to a round $\mbb S^n$ as $t \to 0$. Bourni, Langford and Tinaglia  recently constructed the first example of a compact ancient solution which is interior collapsing \cite{Bourni17}. 
    
In codimension one, a great deal is known about the mean curvature flow under natural curvature conditions such as convexity. In the present work we build mainly on a recent paper by Huisken and Sinestrari \cite{Huisken2015}, where estimates proven earlier by Huisken \cite{Huisken1984} are used to give several characterisations of the shrinking sphere amongst convex ancient solutions (similar results were proven in \cite{Haslhofer2016} by other methods). In particular, Huisken and Sinestrari show that any mean convex ancient solution with uniformly pinched second fundamental form, 
\begin{equation}
\label{eqn_hyppinch}
 h_{ij} \geq \e H g _{ij}, \qquad\e>0,
 \end{equation} 
 must be a family of shrinking spheres.  Similar results also hold for the Ricci flow \cite{Daskalopoulos2012}, \cite{Brendle2011}, and for a large class of fully nonlinear flows of hypersurfaces \cite{Lang17}. 
 
 In higher codimensions, far less is known about the mean curvature flow in general. This is due to the presence of the normal curvature, which complicates the structure of evolution equations for geometric quantities along the flow, and the fact that the second fundamental form is now a normal bundle-valued tensor, so that no obvious notion of convexity is available. We continue here the study of a natural pinching condition, $|h|^2 \leq c |H|^2$, which Andrews and Baker \cite{Andrews2010} (cf. \cite{Hu87}) have successfully employed as an alternative to convexity in higher codimensions. They demonstrated that for values $c \leq \min\{\frac{4}{3n}, \frac{1}{n-1}\}$, this condition is preserved by the flow, and solutions satisfying it flow to round spheres. To motivate this condition, we note that in Euclidean space, a mean convex hypersurface satisfying $|h|^2 \leq \frac{1}{n-1}|H|^2$ is automatically weakly convex (see Lemma \ref{lem_alg1}), while a general submanifold satisfying this condition has nonnegative sectional curvature \cite{Chen93}, and in fact, positive curvature operator (this we prove in Theorem \ref{thm_riempinch} below).
 
Our main result is a high codimension analogue of the sphere characterisation of Huisken and Sinestrari \cite{Huisken2015}, assuming the Andrews-Baker condition in place of \eqref{eqn_hyppinch}.
\begin{theorem}
\label{thm_main}
Let $F: \mathcal M^n \times (-\infty,0) \to \mathbb{R}^{n+k}$, $n \geq 2$, be a compact ancient solution to mean curvature flow with non-vanishing mean curvature vector. Suppose there is a constant $\varepsilon >0$ such that for each $t <0$, the second fundamental form of $\mc M_t$ satisfies
 \begin{equation}
\label{eqn_pinching}
|h|^2 - c_n |H|^2 \leq - \varepsilon |H|^2,
\end{equation}
where
\begin{equation*}
\arraycolsep=1.4pt\def\arraystretch{2.2}
c_n := \left\{ 
\begin{array}{cc}
\dfrac{4}{3n},& \text{ if $n=2,3$},\\
\dfrac{1}{n-1},& \quad \text{ if $ n\geq 4$.} 
\end{array}
\right.\end{equation*}
 Then  $\mathcal M_t$ is a family of shrinking spheres.
\end{theorem}

The constant $c_n$ is optimal for $n \geq 4$, since, for $k=1$ and every $n\geq 2$, Haslhofer and Hershkovits \cite{Haslhofer2016} have constructed an ancient solution other than the shrinking sphere which satisfies $|h|^2 \leq \frac{1}{n-1}|H|^2$. The values $c_{2}$ and $c_3$ come out of the analysis in \cite{Andrews2010}, rather than geometric considerations, and might be improved. There is however an immersion of the Veronese surface into $\mbb R^5$ which shrinks homothetically under the mean curvature flow and satisfies $|h|^2 = \frac{5}{6} |H|^2$, so one cannot hope to do better than $c_2 = \frac{5}{6}$. 

The paper is arranged as follows. In Section 3, we show that any complete submanifold of Euclidean space with bounded, non-vanishing mean curvature, and which satisfies the pinching condition \eqref{eqn_pinching}, must be compact. This is a natural high codimension generalisation of a theorem of Hamilton \cite{Hamilton1994}, which asserts the compactness of complete hypersurfaces satisfying \eqref{eqn_hyppinch}. In Section 4 we use this compactness result to prove Theorem \ref{thm_main},
as well as a further characterisation of the shrinking sphere as the only weakly pinched ancient solution with type-I curvature growth. We then apply our results to provide an alternate proof of the convergence theorem of Andrews and Baker which does not make use of Stampacchia iteration or a gradient estimate.  Finally, in Section 6, we prove a classification analogous to Theorem \ref{thm_main} for high codimension ancient solutions in the sphere.

We would like to thank Mat Langford for many helpful discussions which have benefited this work.
  
\section{Preliminaries}
Let $\mathcal M^n$ be a smooth, immersed submanifold of a Riemannian manifold $\mathcal N^{n+k}$. Denote the curvature operator of $\mathcal N$ by $\bar R$. We will usually take $\mathcal N$ to be Euclidean space - only in the final section do we consider $\mathcal N = \mathbb{S}^{n+k}$. We work in local orthonormal frames for the tangent bundle $  T \mc  M $ and normal bundle $ N \mc M$, denoted by $\{e_i\}$ and $ \{ \nu _ \alpha\} $ respectively. When working in such bases, unless otherwise specified, we will sum over repeated indices whether they are raised or lowered. For example, we may write the mean curvature vector as
\begin{align*}
H = \tr_ g h = g^{ij} h _{ ij }  = h_{i i} = {h_i}^i = g^{ij} {h _{ij}} ^ \alpha \nu_\alpha = h_{ii\alpha} \nu _\alpha. 
\end{align*}
We can then write familiar equations such as Codazzi's equation as 
\begin{align*}
\nabla _ i h_{jk \alpha} -\nabla_j h_{ ik \alpha} = \bar R _{  ijk\alpha}. 
\end{align*}
Gauss' equation is given by 
\begin{align*}
R _{ijkl} & = h _{ik \alpha} h _{ jl \alpha} - h _{ jk \alpha} h _{il \alpha} + \bar R _{ ijkl},
\end{align*}
and for the normal curvature we have
\begin{align*}
R^\perp _{ ij \alpha \beta} = h _{i p \alpha } h_{j p \beta} - h _{ j p \alpha} h_{i p\beta} + \bar R _{ij\alpha\beta}. 
\end{align*}
In fact, the normal curvature depends only on the traceless second fundamental form. Writing
\begin{align*}
h_{ij \alpha } = \circo h _{ ij \alpha} + \frac{ H _ \alpha}{ n}  g _{ij},
\end{align*}
we have 
\begin{align*}
R^\perp _{ ij \alpha \beta} = \circo h _{i p \alpha } \circo h_{j p \beta} - \circo h _{ j p \alpha} \circo h_{i p\beta} + \bar R _{ij\alpha\beta} .
\end{align*}
\subsection{Evolution Equations}

Equations for the evolution of all relevant geometric quantities along the flow are computed in detail in \cite{Andrews2010}. For the second fundamental form we have
\begin{align}\label{eqn_evolSFF}
\nabla _{ \partial _t} h _{ ij \alpha} & = \Delta h_{ij \alpha} + h_{ij \beta}  h_{ pq \beta} h_{pq \alpha} + h_{iq \beta}  h_{ qp \beta} h_{pi \alpha} + h _{ jq\beta}  h _{ qp \beta} h_{pi \alpha} \\
&\;\;\;\;- 2 h_{ip \beta}  h_{jq \beta} h_{pq \alpha}, \notag
\end{align}
and taking the trace with respect to $i$ and $j$, 
\begin{align}\label{eqn_evolmean}
\nabla _{ \partial _t} H _ {\alpha} & = \Delta H _ \alpha + H _ \beta  h_{ pq \beta } h_{ pq \alpha}.
\end{align}
The equations for $ |h|^2$ and $|H|^ 2 $ are then
\begin{align}
\partial _ t |h| ^ 2 & = \Delta |h| ^ 2 - 2|\nabla h| ^ 2 + 2 \sum_{ \alpha, \beta} \left( \sum _{ i , j } h _{ij \alpha} h _{ ij \beta} \right )^ 2 \\
&\hspace{3.02cm} + 2 \sum_{i,j,\alpha,\beta} \left( \sum_p h_{ ip\alpha} h _{ jp \beta} - h _{ jp \alpha} h _{ ip \beta }\right) ^ 2 \notag,
\end{align}
and
\begin{align}
\partial _ t |H|^ 2 & = \Delta |H| ^ 2 - | \nabla H| ^ 2 + 2 \sum_{ i, j } \left( \sum _\alpha H _ \alpha h _{ ij \alpha}\right)^ 2.
\end{align}
The last term in the evolution equation for $|h|^2$ can be expressed purely in terms of the normal curvature,
\begin{align*}
\sum_{i,j,\alpha,\beta} \left( \sum_p h_{ ip\alpha} h _{ jp \beta} - h _{ jp \alpha} h _{ ip \beta }\right) ^ 2 =\sum_{i,j,\alpha,\beta} \left( \sum_p \circo h_{ ip\alpha} \circo h _{ jp \beta} - \circo h _{ jp \alpha} \circo h _{ ip \beta }\right) ^ 2 = | R ^ \perp | ^ 2 .
\end{align*}
We will find it convenient to denote the reaction terms above by
\begin{align*}
R _1 & = \sum_{\alpha,\beta} \left( \sum _{ i ,j } h _{ ij \alpha} h _{ ij \beta} \right) ^ 2 + | R ^ \perp | ^ 2, \\
R_2 & = \sum_{i,j} \left( \sum _ {\alpha} H _ \alpha h_{ij \alpha}\right)^ 2 . 
\end{align*}

\subsection{Preservation of pinching}
We consider the quadratic quantity
\begin{align}
\mc Q = | h| ^ 2 + a - c |H| ^ 2 
\end{align} 
where $ c$ and $ a$ are positive constants. Combining the evolution equations for $ |h| ^ 2 $ and  $| H|^ 2$ yields
\begin{align}
\label{eqn_pinchpres}
\partial_t \mc Q & = \Delta \mc Q - 2 ( |\nabla h |^2 - c |\nabla H| ^ 2 ) + 2 (R_1 - c R _2).
\end{align} 
The gradient estimate
\begin{align*}
| \nabla h|\geq \frac{ 3}{n+2} | \nabla  H |^ 2,  
\end{align*}
which is proven as in Hamilton \cite{Hamilton1982} and Huisken \cite{Huisken1984}, shows
that the gradient terms in \eqref{eqn_pinchpres} are strictly negative if $ c < \frac{ 3}{n+2}$. Careful estimating shows that for $c < \frac{4}{3n}$ we also have $R_1 - cR_2 < 0$ (see \cite{Andrews2010}), so by the maximum principle:

\begin{lemma}
Let $F:\mc M ^ n \times[0,T)\rightarrow \R^{n+k}$ be a solution to the mean curvature flow such that $\mc M_0$ satisfies 
\begin{align*}
|h|^2 + a \leq  c |H|^2 
\end{align*}
for some $a > 0$ and $ c \leq \frac{4}{3 n }$. Then this condition is preserved by the mean curvature flow.  
\end{lemma}

As a consequence, we see that the flow preserves both $|H| >0$ and $|h|^2 - c_n |H|^2 \leq - \varepsilon |H|^2$. 

\section{Curvature Pinching}

We begin with a purely algebraic calculation, which in particular shows that sufficiently pinched symmetric matrices are positive/negative definite. 

\begin{lemma}
\label{lem_alg1}
Let $ B$ be a symmetric matrix with eigenvalues $\kappa $. For any two eigenvalues $\kappa_1$, $\kappa_2$, there holds
\begin{align*}
|B|^2 - \frac{1}{n-1} (\tr B) ^2 & = -2 \kappa _1 \kappa_2 + \left ( \kappa_1 + \kappa_2 - \frac{1}{n-1} \tr B \right)^2 + \sum_{l=3}^n \left( \kappa_ l - \frac{1}{n-1} \tr B \right) ^ 2. 
\end{align*}
\end{lemma}
\begin{proof}
We expand the right-hand side
\begin{align*}
 \left ( \kappa_1 + \kappa_2 - \frac{1}{n-1} \tr B \right)^2 & = \kappa_1^2 + 2 \kappa_1\kappa_2 + \kappa_2 ^ 2 - 2\frac{\kappa_1 + \kappa_2} { n-1} \tr B  + \frac{1}{(n-1)^2} (\tr B)^2,
\end{align*}
and note that
\begin{align*}
\sum_{l=3}^n \left( \kappa_ l - \frac{1}{n-1} \tr B \right) ^ 2 & = \sum_{l=3} ^ n \left( \kappa_ l ^ 2 - 2 \frac{ \kappa_l}{n-1} \tr B + \frac{1}{ (n-1)^2}(\tr B)^2 \right)\\
& = \sum _{ l=3}^ n \kappa_l^2 - \frac{2 \tr B}{n-1} \sum_{l=3}^n \kappa_l + \frac{ n-2}{(n-1)^2} (\tr B)^2.
\end{align*} 
This gives
\begin{align*}
-2 \kappa _1 \kappa_2 &+ \left ( \kappa_1 + \kappa_2 - \frac{1}{n-1} \tr B \right)^2 + \sum_{l=3}^n \left( \kappa_ l - \frac{1}{n-1} \tr B \right) ^ 2 \\
&=\kappa_1^2 + \kappa_2 ^ 2 - \frac{ 2 ( \kappa_1 + \kappa_2) }{ n-1} \tr B + \frac{1}{ (n-1) ^ 2 } (\tr B)^ 2 + \sum_{l=3} ^ n \kappa_l ^ 2\\
&\;\;\;\; - \frac{2}{ n-1} \sum_{l=3} ^ n \kappa_ l \tr B + \frac{n-2}{ (n-1)^2} (\tr B)^2 \\
&= |B|^2 - \frac{2}{n-1} \tr B^2 + \frac{ n-1} { (n-1)^2} \tr B^2 = | B|^2 - \frac{1}{n-1} (\tr B)^2.
\end{align*}
\end{proof}

Hence, we see that if $|B|^2 - \frac{1}{n-1} (\tr B) ^2 \leq 0$, then all the eigenvalues of $ B$ have the same sign. In particular, if $\tr B > 0$ then $B$ is positive definite. This allows us to pull the pinching condition \eqref{eqn_pinching} back to an intrinsic curvature condition on $\mc M^n$.

\begin{lemma}
\label{thm_riempinch}
An immersed submanifold $\mc M^n \subset \mathbb{R}^{n+k}$, $n\geq 2$, whose second fundamental form satisfies
\[|h|^2 - \frac{1}{n-1}|H|^2\leq -\varepsilon |H|^2\]
for some $\varepsilon > 0 $ has curvature operator pinched by
\[\mc R \geq \frac{\varepsilon}{2} |H|^2I.\]
\end{lemma}

\begin{proof}
It suffices to work over a single point $p \in \mc M$. Using the Gauss equation, we split the curvature operator into components
\begin{align*}
\mc R(x \wedge y, u \wedge v) &= h(x,u) \cdot h(y,v)- h(x,y)\cdot h(u,v)\\
		&= \sum_\alpha h_\alpha(x,u)h_\alpha(y,v) - h_\alpha(x,y)h_\alpha(u,v)\\
		 & =: \sum_\alpha \mc R_\alpha(x\wedge y, u \wedge v).
\end{align*}
For any fixed index $\alpha$ we can choose an orthonormal frame $\{e_i\}$ for $T_p \mc M$ which diagonalises $h_\alpha$, in which case the bivectors $e_i \wedge e_j$ with $i\not = j$ diagonalise $\mc R_\alpha$. The corresponding eigenvalues are (no summation)
\[ \mc R_\alpha (e_i \wedge e_j, e_i \wedge e_j) = h_{\alpha i i} h_{\alpha j j},\]
so that by Lemma \ref{lem_alg1}, 
\[\mc R_\alpha(\omega, \omega) \geq \frac{1}{2}\left(\frac{1}{n-1}|H_\alpha|^2 - |h_\alpha|^2\right) |\omega|^2\]
for any $\omega \in \bigwedge^2 T_p \mc M$. Taking the sum, we obtain
\[\mc R(\omega, \omega) \geq \frac{1}{2}\sum_{\alpha = 1}^k\left( \frac{1}{n-1}|H_\alpha|^2 - |h_\alpha|^2 \right)|\omega|^2\geq \frac{\varepsilon}{2} |H|^2 |\omega|^2.\]
\end{proof}

As a direct application of the above estimate and a theorem of Ni-Wu \cite{Ni2007}, we obtain a high codimension version of the main theorem in \cite{Hamilton1994}.

\begin{corollary}\label{thm_compact}
Any complete immersed submanifold $\mc M^n$ of $\mathbb{R}^{n+k}$ with bounded, non-vanishing mean curvature vector, and which has second fundamental form pinched by
\[|h|^2 - \frac{1}{n-1}|H|^2 < -\varepsilon |H|^2\]
for some $\varepsilon > 0$, is compact. 
\end{corollary}

\begin{proof}
Since $|H|$ is bounded, the pinching ensures that $|h|$, and therefore the full curvature operator, are also bounded. Taking traces of the Gauss equation shows that the scalar curvature $\Sc$ is given by 
\[\Sc = |H|^2 - |h|^2,\]
so applying Theorem \ref{thm_riempinch}, we see that the curvature operator is pinched by
\[ \mc R \geq \varepsilon \, \Sc \, I.\]
By a result of Ni-Wu \cite{Ni2007}, $\mc M$ must then be compact.
\end{proof}

\section{Ancient Solutions in Euclidean Space} 

The following theorem is due to Huisken and Sinestrari \cite{Huisken2015} in case $k=1$. With the estimates of Andrews and Baker in place, the proof in higher codimensions is the same.  

\begin{theorem}
\label{thm_ancmain}
Let $F: \mc M^n \times (-\infty, 0) \to \mathbb{R}^{n+k}$ be a compact ancient solution to mean curvature flow satisfying the pinching condition $|h|^2 - c_n |H|^2 \leq - \varepsilon |H|^2$ for some $\varepsilon > 0$. If, in addition, the area $\mu(\mc M_t) = \int_{\mc M_t} d\mu_g$ satisfies the decay condition
\begin{equation}
\label{eqn_areadecay}
\mu(\mc M_t) \leq c|t|^r, \qquad t\leq -T,
\end{equation}
for some positive constants $c$, $r$ and $T$ independent of $t$, then $\mc M_t$ is a family of shrinking spheres. 
\end{theorem}
\begin{proof}
We show that for small enough $\sigma >0$, the function 
\[f_\sigma := \frac{| \ho|^2}{|H|^{2(1-\sigma)}}\]
vanishes identically along the flow. It was shown in \cite[Lemma 5]{Andrews2010} that for pinched solutions, there are constants $p \gg 1$ and $\sigma \sim \frac{1}{\sqrt{p}}$ depending only on $n$ and $\varepsilon$ such that
\[\frac{d}{dt} \int f_\sigma^p \, d\mu_g \leq - \int |H|^2 f_\sigma^p \, d\mu_g.\]
Setting $\gamma := 1 + \frac{2}{\sigma p}$, we have $f_\sigma^{\gamma p} \leq |H|^2 f_\sigma^p$, and H\"{o}lder's inequality implies that
\[\int f_\sigma^p \, d\mu_g \leq \mu(\mc M_t)^\frac{2}{\gamma \sigma p} \left( \int f^{\gamma p}\, d\mu_g\right)^ \frac{1}{\gamma} \leq \mu(\mc M_t)^\frac{2}{\gamma \sigma p} \left( \int |H|^2f^{p}\, d\mu_g\right)^\frac{1}{\gamma},\]
and in turn,
\[\frac{d}{dt} \int f_\sigma^p \, d\mu_g \leq - \mu(\mc M_t)^{-\frac{2}{\sigma p}}\left( \int  f_\sigma^p \, d\mu_g\right)^\gamma.\]
Let $\varphi \:= \int f_\sigma^p \, d\mu_g$ and suppose that $\varphi(s) >0$ for some $s \in (-\infty, -T]$. Since $\varphi$ may not increase in time, this implies that $\varphi(t) > 0$ for all $ t\in (-\infty, s]$, and we have 
\[\frac{1}{1-\gamma} \frac{d \varphi^{1-\gamma}}{dt} = \varphi^{-\gamma} \frac{d}{dt} \varphi \leq - \mu(\mc M_t)^{-\frac{2}{\sigma p}} \leq - c|t|^{-\frac{2r}{\sigma p}}, \qquad t \leq s.\]
Integrating in time then yields
\begin{align*}
\varphi^{1-\gamma}(s)  &\geq  \varphi^{1-\gamma}(t)+\frac{2c}{\sigma p} \int_t^s |\tau| ^{-\frac{2r}{\sigma p}}\,d\tau \\
		&\geq \frac{2c}{\sigma p-2r} \left(|t|^{1-\frac{2r}{\sigma p}} - |s|^{1-\frac{2r}{\sigma p}}\right).
\end{align*}
Since $\sigma p \sim \sqrt{p}$ we may choose $p$ so large that $\sigma p >2r$, in which case the right-hand side of the last inequality becomes unbounded as $t \to -\infty$. This is a contradiction, so it must be the case that $\varphi(s) = 0$ for all $s \in (-\infty, -T]$, and hence for all $s <0$. 
\end{proof}

In other words, a pinched ancient solution with sufficiently slow area decay must be a family of shrinking spheres. To control the area of pinched, codimension one solutions, Huisken and Sinestrari \cite{Huisken2015} use a Gauss-Bonnet-type result. A similar approach works in all codimensions for flows of surfaces. 

\begin{proof}[Proof of Theorem \ref{thm_main}  ($n=2$)]
Theorem \ref{thm_riempinch} says that the Gauss curvature $\kappa$ of $\mc M_t$ satisfies
\[\kappa \geq \frac{\varepsilon}{2} |H|^2,\]
and since $\mc M_t$ evolves smoothly and shrinks to a round point as $t \to 0$ by \cite{Andrews2010}, it is diffeomorphic to $\mbb S^2$ at every fixed time. We may therefore apply the Gauss-Bonnet theorem to conclude that 
\[\int |H|^2 \, d\mu_g \leq \frac{2}{\varepsilon} \int \kappa \, d\mu_g = \frac{4\pi}{\varepsilon}.\]
Substituting into the area decay formula then yields
\[-\frac{d}{dt} \mu(\mc M_t) = \int |H|^2 \, d\mu_g \leq \frac{4\pi}{\varepsilon},\]
which we integrate in time to obtain
\[\mu(\mc M_t) \leq - \frac{4\pi}{\varepsilon}t. \]
Theorem \ref{thm_ancmain} can then be applied to finish. 
\end{proof}

It is not clear that this type of argument generalises to higher dimensions and codimensions, however we are still able to prove an area decay estimate in this setting by applying Corollary \ref{thm_compact}. We will find it convenient to introduce a dichotomy analogous to the one used when classifying finite-time singularities - we say that an ancient solution is of type I if there are positive constants $C$ and $T$ such that 
\[\max_{x \in \mc M} |H(x,t)| \leq \frac{C}{\sqrt{-t}}, \qquad t \leq -T,\]
and of type II in case
\begin{align*}
\limsup_{t \to -\infty}\, \max_{x \in \mc M} \sqrt{-t} |H(x,t)| = \infty.
\end{align*}

\begin{proof}[Proof of Theorem \ref{thm_main} ($n \geq 3$)]
We proceed by ruling out type-II blow-downs, and then showing that any type-I solution has slow enough area decay to apply Theorem \ref{thm_ancmain}.

Suppose that $F$ is of type II. To derive a contradiction, we choose $(x_j,t_j) \in \mathcal M \times [-j, 0)$, $j \in \mathbb{N}$, so that 
\[-t_j|H(x_j, t_j) |^2= \max_{(x,t) \in\, \mc M \times [-j,0)}-t|H(x,t)|^2.\]
We set $L_j = |H(x_j,t_j)|^2$, and note that the type-II condition implies
\[t_j \to -\infty, \qquad -t_j L_j \to \infty.\]
Following Huisken-Sinestrari \cite{Huisken2015} (cf. \cite{Hamilton1995a}, \cite{HuSi99a}), we consider the sequence of rescaled and translated flows
\[F_j(\cdot, \tau) := \sqrt{L_j} (F( \cdot\,, \tau L_j^{-1} + t_j) - F(x_j,t_j)),  \qquad \tau \in (-\infty, -t_jL_j).\]
Let $H_j$ denote the mean curvature vector corresponding to $F_j$ and observe that
\[1 = |H_j(x_j, 0)|= \max_{x \in \mathcal M} |H_j(x,0)|.\]
Our definition of $(x_j,t_j)$ ensures that for $\tau \in (0, -t_jL_j)$,
\[-(\tau L_j^{-1} + t_j) |H(x,\tau L_j^{-1} + t_j) |^2 \leq -t_j L_j,\]
so we have the bound
\[ |H_j(x,\tau)|^2 \leq \frac{t_j}{\frac{\tau}{L_j} + t_j}.\]
This implies that $|H_j(\cdot, \tau)|^2 \leq 2$ for all $\tau \in (0, -\frac{1}{2}t_j L_j)$, which combined with the pinching assumption provides a uniform bound for the full second fundamental form. We may therefore extract a subsequence of rescalings converging to a complete solution $F_\infty$ of mean curvature flow defined on the time interval $(1, \infty)$, and with mean curvature satisfying $0<|H_\infty| \leq 2$ . This is a contradiction - the pinching condition \eqref{eqn_pinching} is scale invariant and carries over to the limit, but forces compactness of the solution on any timeslice by Theorem \ref{thm_riempinch}, so $F_\infty$ must become singular in finite time.

We are left with the possibility that $F$ has type-I curvature growth on (without loss of generality) the time interval $(-\infty, -1]$. That is,
\begin{equation}
\max_{x \in \mc M} |H(x,t)|^2 \leq - \frac{C}{t}, \qquad t \leq -1, 
\end{equation}
for some $C>0$. The area decay formula then yields
\[-\frac{d}{dt} \mu(\mc M_t) = \int |H|^2 \, d\mu_g \leq -\frac{C}{t}\mu(\mc M_t)\] 
which we integrate to obtain
\[\mu(\mc M_t) \leq \mu(\mc M_{-1})|t|^{C}, \qquad t\leq -1. \]
Hence $\mc M_t$ is totally umbilic for all times by Theorem \ref{thm_ancmain}.
\end{proof}

Theorem \ref{thm_main} implies the following further characterisation of the shrinking sphere. Huisken and Sinestrari prove an analogous result for $k =1$ \cite{Huisken2015}, and we adapt their argument. 

\begin{theorem}
Let $F: \mathcal M^n \times (-\infty,0) \to \mathbb{R}^{n+k}$, $n \geq 2$, be a compact ancient solution to mean curvature flow with non-vanishing mean curvature vector which is weakly pinched,
\[|h|^2 - c_n|H|^2 \leq 0,\]
and has type-I curvature growth. Then $\mc M_t$ is a family of shrinking spheres. 
\end{theorem}
\begin{proof}
Recall that the type-I condition says
\[|H| \leq \frac{C}{\sqrt{-t}}, \qquad t \leq -T < 0.\]
This implies the bound
\begin{align}
|F(p,t) - F(p,-T)| & \leq \int_t^{-T} |H(p,\tau)| \, d\tau \leq 2C\sqrt{ - t}, \qquad t \leq -T
\end{align}
for any $p \in \mc M$, so for any pair of points $p, q \in \mc M$ there holds
\begin{align}
\label{eqn_TIdist}
|F(p,t) - F(q,t)| &\leq 4C\sqrt{-t} + |F(p, -T) - F(q,-T)| \notag \\
	&\leq 5 C \sqrt{-t}
\end{align}
as long as $t  \leq -C^{-1} \left( \sup_{p, q \in \mc M} |F(p, -T) - F(q,-T)|\right)^2$.
  
 Suppose now for a contradiction that $F$ is not uniformly pinched. That is, we assume there is a sequence of times $t_j \to -\infty$ and points $x_j \in  \mc M$ such that 
\[\lim_{j \to \infty} \frac{|h(x_j,t_j)|^2 - \frac{1}{n-1}|H(x_j, t_j)|^2}{|H(x_j, t_j)|^2} = 0.\]
We define a sequence of rescaled flows,
\[F_j(x,\tau) := \frac{1}{\sqrt{-t_j}} F\left(x,  -t_j\tau\right), \qquad \tau \in [-2, -1], \]
which, as a consequence of \eqref{eqn_TIdist}, all take values in a single compact subset of $\mathbb{R}^{n+k}$. The type-I condition provides a uniform upper bound for the functions $|H_j|$, which translates to an upper bound for the sequence $|h_j|$ via pinching, so the sequence $F_j$ converges to a compact, weakly pinched solution $F_\infty$ defined on $\mc M \times [-\frac{3}{2},-1]$. The mean curvature of the limit $H_\infty$ satisfies $|H_\infty| \geq 0$, so \eqref{eqn_pinchpres} and the strong maximum principle imply that $|H_\infty| > 0$ for $t > -\frac{1}{2}$. Our choice of sequence then implies the existence of an $x_\infty \in \mc M$ such that
\[|h_\infty(x_\infty, -1)|^2 = \frac{1}{n-1}|H_\infty(x_\infty, -1)|^2,\]
in which case another application of the strong maximum principle to \eqref{eqn_pinchpres} yields
\begin{equation}
\label{eqn_hinfty}
|h_\infty|^2 \equiv \frac{1}{n-1}|H_\infty|^2.
\end{equation}
This forces the gradient and reaction terms in \eqref{eqn_pinchpres} to vanish on all of $\mc M \times [-\frac{1}{2}, -1]$, which implies that $F$ is a shrinking sphere solution (see \cite{Baker2011}), contradicting \eqref{eqn_hinfty}.
\end{proof}

\subsection{Convergence to round points}

For dimensions $n \geq 3$ our results give an alternate proof of the convergence theorem due to Andrews and Baker, which says that pinched solutions shrink to round points. Consider a compact solution $F : \mc M^n \times [0,T) \to \mathbb{R}^{n+k}$ such that $T$ is maximal and the pinching condition \eqref{eqn_pinching} is satisfied at $t =0$. For convenience we assume $F$ is scaled so that $T >1$. That the pinching condition is preserved by the flow follows from the maximum principle applied to \eqref{eqn_pinchpres}. If $F$ undergoes a type-II singularity as $t \to T$, then we perform a Hamilton blow-up, similar to that in the proof of Theorem \ref{thm_main}. By assumption the quantity
\[(T-t_j )|H(t_j, x_j)|^2 := \max_{(x,t) \in \mc M \times[0,T-\frac{1}{j}]} (T-t)|H(x,t)|^2, \qquad j \in \mathbb{N},\]
blows up as $j \to \infty$. Thus, setting  $L_j = |H(x_j,t_j)|^2$,  the sequence of rescalings
\[F_j(\cdot, \tau) := \sqrt{L_j} (F(\cdot, \tau L_j^{-1} + t_j) - F(x_j, t_j)), \qquad 0 \leq \tau \leq \frac{1}{2}(T- j^{-1}-t_j)L_j ,\]
subconverges to a complete solution defined for times $\tau \in [1,\infty)$, which is uniformly pinched and has bounded, non-vanishing mean curvature. This is a contradiction, since Corollary \ref{thm_compact} implies that the solution is compact on every timeslice, and must therefore become singular in finite time. We conclude that $F$ undergoes a type-I singularity, that is, there is a constant $C > 0$ such that 
\[\max_{x \in \mc M} |H(x,t)| \leq \frac{C}{\sqrt{T-t}}, \qquad t\in [0,T).\]
In this case, we define for $j \in \mathbb{N}$ the sequence of rescalings
\[F_j (\cdot, \tau) := j F(\cdot, j^{-2} \tau + T), \qquad \tau \in [-j^2T, -1],\]
As in the proof of Theorem \ref{thm_main}, the type I assumption provides a uniform radius bound, as well as the curvature bound,
\[\max_{(x,t) \in \mc M \times [-j^2T, -1]} |H_j(x,t)| = \max_{(x,t) \in \mc M \times [0, T- j^{-2}]} \frac{|H(x,t)|}{j} \leq C,\]
so the sequence subconverges to a uniformly pinched ancient solution on the time interval $\tau \in (-\infty, -2]$. Corollary \ref{thm_compact} again ensures compactness on timeslices, so we may apply Theorem \ref{thm_main} to conclude that the limit is a shrinking sphere. 

\section{Pinched Ancient Solutions in the Sphere}

In this section we prove an analogue of Theorem \ref{thm_main} for the mean curvature flow in a spherical background. We consider solutions $F: \mc M^n \times (-\infty, 0 ) \rightarrow   \mbb S ^{n+k}_{R}$, where $ \mbb S ^{n+k}_R$ is the $(n+k)$-sphere of radius $R>0$ with sectional curvature $ K = R^{-2}$. 

Unlike Euclidean space, the sphere contains compact minimal surfaces. The simplest examples are the totally geodesic spheres, such as the equators. Minimal surfaces generate static solutions to the mean curvature flow, which are not only ancient but eternal. Further examples of ancient solutions are the shrinking spherical caps, which flow out of geodesic spheres on the equator and shrink to round points at a pole. The following theorem says that sufficiently pinched ancient solutions must be of one of these two forms (cf. \cite{Huisken2015} Theorem 6.1). 
\begin{theorem}
Let $F : \mc M^n \times (-\infty, 0) \to \mathbb{S}_R^{n+k}$ be a compact ancient solution to the mean curvature flow satisfying $|H| > 0$ for all times. 
\begin{enumerate} 

\item If there is a $\delta >0 $ such that for every $t \in (-\infty, 0)$ there holds
\begin{align*}
|h|^2 - \frac{1}{3} |H|^2 \leq (2-\delta) K, \qquad n =4,
\end{align*}
or 
\[|h|^2 - \frac{1}{n-1} |H|^2 \leq 2K, \qquad n \geq 5,\]
 then $\mc M_t$ is either a shrinking spherical cap or a totally geodesic sphere.
 
 \item If  $  |h|^2 \leq \frac{4}{3n} |H|^2$ for every $ t \in (-\infty, 0)$ then $ \mc M_t$ is a shrinking spherical cap. 
\end{enumerate}
 
\end{theorem}
\begin{proof}
1). Consider the auxiliary function  $f := \abs{\ho}^2/(\abs{H}^2 + b)$ with 
\[b = (1-\varepsilon)Kn (n-1)\] 
and $\varepsilon\in (0,1)$ to be fixed later. The evolution of this function was computed in \cite{Baker2011}, and is given by
\begin{align}
\label{eqn_evolsphere1}
 \partial_t f &=  \Delta f + \frac{2}{|H|^2 + b} \langle \nabla_i |H|^2 , \nabla_i f \rangle \notag \\
 &\;\;\;\; - \frac{2}{|H|^2 + b} \left(|\nabla h|^2 - \frac{1}{n} |\nabla H|^2 - \frac{|\ho|^2}{|H|^2 + b} |\nabla H|^2 \right) \notag\\
 &\;\;\;\; + \frac{2}{|H|^2 + b} \left( R_1 - \frac{1}{n} R_2 - nK|\ho|^2 - \frac{R_2 |\ho|^2}{|H|^2 + b} - \frac{nK|\ho|^2|H|^2}{|H|^2 + b}\right) \notag\\
 &=: \Delta f + \frac{2}{|H|^2 + b} \langle \nabla_i |H|^2 , \nabla_i f \rangle + \I + \II,
\end{align}
with $R_1$ and $R_2$ as above. In all dimensions, the pinching condition implies that
\[|\ho|^2 \leq \frac{1}{n(n-1)}|H|^2 + 2K,\]
and we still have the gradient estimate
\[\frac{3}{n+2}|\nabla H|^2 \leq |\nabla h|^2,\]
so for $\varepsilon$ sufficiently small, 
\begin{align*}
\I  &\leq -\frac{2}{|H|^2 + b}\left( \frac{3}{n+2} - \frac{1}{n} - \frac{1}{n(n-1)} \left( \frac{|H|^2 + 2Kn(n-1)}{ |H|^2 + (1-\varepsilon)Kn(n-1)}\right)\right)|\nabla H|^2\\
  &\leq -\frac{2}{|H|^2 + b}\left( \frac{3}{n+2} - \frac{1}{n} - \frac{3}{n(n-1)} \right)|\nabla H|^2.
\end{align*}
The right-hand side is nonpositive for all $n \geq 4$ and this term can be discarded. Before estimating the reaction terms, we introduce some notation. Around any point in $\mc M$ we choose a local orthonormal frame $\{ \nu_\alpha\}$ for $N \mc M$ such that $\nu_1 = \frac{H}{|H|}$ and a local orthonormal frame $\{e_i\}$ for $T \mc M$ which diagonalises $h_1$. We then write \[\ho_\alpha = h_\alpha - \frac{H_\alpha}{n} g, \qquad \ho_- = \sum_{\alpha > 1} \ho,\]
so that $|\ho|^2 = |\ho_1|^2 + |\ho_-|^2$. Andrews and Baker compute
\[R_2 = |\ho_1|^2 |H|^2 + \frac{1}{n} |H|^4,\]
and then derive the estimate
\[R_1 - \frac{1}{n} R_2 \leq |\ho_1|^4 + \frac{1}{n} |\ho_1|^2 |H|^2 + 4|\ho_1|^2 |\ho_-|^2 + \frac{3}{2}|\ho_-|^4,\]
so we have
\begin{align*}
\II &\leq  \frac{2}{(|H|^2 + b)^2} \bigg((4-n)|\ho_1|^2 |\ho_-|^2|H|^2 + \bigg(\frac{5}{2}-n\bigg)|\ho_-|^4|H|^2 \\
&\hspace{3.3cm}+2b|\ho|^4 - 2K|\ho|^2|H|^2 - bnK|\ho|^2\bigg)\\
&\leq \frac{2}{(|H|^2 + b)^2}\bigg( 2b|\ho|^4 - (2-\vartheta)K|\ho|^2|H|^2- (n-\vartheta)bK|\ho|^2\bigg)- 2\vartheta K f
\end{align*}
for any constant $\vartheta$, which we take to be small and positive. In case $n=4$, we now use the pinching to bound 
\begin{align*}
2b|\ho|^4 &- (2-\vartheta)K|\ho|^2|H|^2- (n-\vartheta)bK|\ho|^2\\
&\leq  -(2\varepsilon - \vartheta)K n(n-1) |\ho|^4\\
&\;\;\;\;\; - ((1-\varepsilon)(n-\vartheta) -(2-\vartheta)(2-\delta)) K^2n(n-1)|\ho|^2,
\end{align*}
and observe that the right-hand side can be made nonpositive by taking $\vartheta = 2\varepsilon$ sufficiently small. When $n > 4$, this is possible even for $\delta =0$, so in all dimensions we obtain
	\begin{equation}\label{eqn: evol eqn f_sigma 1 sphere}
		\begin{split}
				\p_t f &\leq \Delta f + \frac{ 2 }{ \abs{H} + b} \big\langle \nabla_i\abs{H}^2, \nabla_i f \big\rangle - 2\vartheta K f. \end{split}
	\end{equation}

Assume now that there is a time $t_1 \in (-\infty, 0)$ such that $\mc M _ {t_1} $ is a sphere. This implies that $ f \not \equiv 0$ on $ \mc M _ {t_{1}}$, so by \eqref{eqn: evol eqn f_sigma 1 sphere},
\begin{align*}
0< \max_{\mc M _ {t_1} } f \leq e^{ - 2 \vartheta K (t_1-t)} \max_{\mc M _t} f 
\end{align*}
for all $ t< t_i$. It follows that $ \max _{ \mc M _t} f \rightarrow \infty$ as $t \to -\infty$, but our pinching condition implies
\begin{align*}
|h| ^2 - \frac{1}{n} |H|^2 &\leq \frac{1}{n(n-1)} |H|^2 + 2  K \leq  H^2 + b, 
\end{align*} 
so that $ f \leq 1$ for all times. This is a contradiction, so $\mc M_t$ must be totally umbilic for all $t < 0$, and is therefore either a shrinking spherical cap or totally geodesic sphere. 

2) We now set $b = 0$ in the definition of $f$. Proceeding exactly as before, we find that the pinching condition $|h|^2 \leq \frac{4}{3n}|H|^2$ is exactly what is required to ensure that $\II \leq -4nkf$, and also implies that $\I \leq 0$. The same contradiction argument used above then shows that $\mc M_t$ is totally umbilic for all times, and since the pinching rules out geodesic spheres, $\mc M_t$ must be a shrinking spherical cap. 
\end{proof}

\bibliographystyle{plain}
\bibliography{ancient}
\end{document}